\documentclass[a4paper, reqno]{amsart}
\usepackage[utf8]{inputenc}
\usepackage[ngerman, english]{babel}
\usepackage{graphicx}
\usepackage{svg}
\usepackage{url}
\usepackage{color,soul}
\definecolor{Azure}{rgb}{0.2,0.8,1}
\setulcolor{Azure}

\graphicspath{{images/}}
\usepackage{amsmath, amssymb}
\usepackage{amsthm, thmtools}

\usepackage{enumerate}
\usepackage{nicematrix}
\usepackage{marvosym}

\usepackage{hyperref} 
\usepackage{cleveref}

\usepackage{csvsimple}
\usepackage{tabularx}
\usepackage{algorithm}
\usepackage{algorithmic}

\declaretheorem[name=Theorem]{theorem}
\numberwithin{theorem}{section}
\declaretheorem[sibling=theorem, name=Lemma]{lemma}
\declaretheorem[sibling=theorem, name=Proposition]{prop}

\declaretheorem[sibling=theorem, name=Definition]{defi}

\declaretheorem[sibling=theorem, name=Corollary]{cor}
\declaretheorem[sibling=theorem, name=Remark]{remark}
\declaretheorem[sibling=theorem, name=Notation]{notation}

\declaretheorem[sibling=theorem, name=Question]{question}

\declaretheorem[sibling=theorem, name=Reminder]{reminder}

\declaretheorem[sibling=theorem, name=Conjecture]{conjecture}
\newcommand{\Q}{\mathbb{Q}}
\newcommand{\R}{\mathbb{R}}
\newcommand{\C}{\mathbb{C}}
\newcommand{\N}{\mathbb{N}}

\newcommand\Z{\mathbb Z}

\newcommand{\E}{\mathbb E}

\DeclareMathOperator{\im}{im}

\DeclareMathOperator{\subjto}{s. t.}

\DeclareMathOperator{\GL}{GL}

\DeclareMathOperator{\GM}{GM} % gaussian moment variety
\newcommand{\ot}{\otimes}

\newcommand{\duonomial}[2]{\widehat{\widehat{\binom{#1}{#2}}}}
\newcommand{\lf}{\mathcal{F}_1}
\newcommand{\quadf}{\mathcal{F}_2}

\usepackage{pifont}

\begin{document}

%% ================= Abstract for paper ==================================
\begin{abstract} 
	We resolve most cases of identifiability from sixth-order moments for Gaussian mixtures on spaces of large dimensions. Our results imply that the parameters of a generic mixture of $ m\leq\mathcal{O}(n^4) $ Gaussians on $ \R^n $ can be uniquely recovered from the mixture moments of degree $ 6 $. The constant hidden in the $ \mathcal{O} $-notation is optimal and equals the one in the upper bound from counting parameters. We give an argument that degree-$ 4 $ moments never suffice in any nontrivial case, and we conduct some numerical experiments indicating that degree $ 5 $ is minimal for identifiability. 
\end{abstract}

\title[Gaussian Mixture Moment Identifiability]{Gaussian Mixture Identifiability\\ from degree 6 moments}
\date{\today}
%\date{February 8, 2023}
\author[Blomenhofer]{Alexander Taveira Blomenhofer}
\thanks{Centrum Wiskunde \& Informatica, Amsterdam}
\address{CWI, Networks \& Optimization, Amsterdam, Science Park 123, NL-1098 XG.}
\email{atb@cwi.nl}
\urladdr{cwi.nl/en/people/filipe-alexander-taveira-blomenhofer}
\maketitle

\section{Introduction}\label{sec:introduction}

Gaussian mixtures are a fundamental distributional model for machine learning and Data Science applications. They have a tremendous range of applications, encompassing Machine learning primitives such as clustering (\cite{Reynolds_Rose_1995a},\cite{Permuter_Francos_Jermyn_2003a},\cite[Notebook 5.12]{Vanderplas_2016python},\cite{mixture_models_2020}) and subspace learning (\cite{Hong_Malinas_Fessler_Balzano_2018},\cite{Lipor_Balzano_2017},\cite{Wang_Wipf_Ling_Chen_Wassell_2015},\cite{Hegde_Indyk_Schmidt_2016a}). 
Historically, Gaussian mixtures emerged as an object of study due to Pearson \cite{Pearson_1900}, who tried to gather evidence for the theory of evolution by separating distinct crab species. In a time where DNA tests were not yet accessible, he inferred, based on a Gaussian mixture model, that his sample of crabs likely consisted of more than one species.  

Ca. 120 years later, there is a vast ocean of research on the topic of Gaussian mixtures (e.g.,  \cite{Dasgupta_1999a},\cite{Sanjeev_Kannan_2001a},\cite{Dasgupta_Schulman_2007a},\cite{Moitra_Valiant_2010},\cite{Kalai_Moitra_Valiant_2010},\cite{Hsu_Kakade_2013a},\cite{Anderson_Belkin_Goyal_Rademacher_Voss_2014a},\cite{Regev_Vijayaraghavan_2017a},\cite{Liu_Moitra_2021},\cite{Amendola_Faugere_Sturmfels_2016},\cite{Amendola_Ranestad_Sturmfels_2018},\cite{Curto_DiDio_2022},\cite{DiDio_2023}), fueled by the craze on machine learning and big data. Many fundamental aspects of Gaussian mixtures are still very poorly understood, including identifiability and efficient parameter recovery. 

A mixture of $m$ Gaussians is sampled as follows: From a box containing the $m$ Gaussians $ \mathcal{N}(\mu_1, \Sigma_1),\ldots,\mathcal{N}(\mu_m, \Sigma_m) $, draw one of them at random. Then, sample the Gaussian that was drawn from the box. The probability to draw the $ i $-th Gaussian is called the $ i $-th \emph{mixing weight} $ \lambda_i $. 
We denote such a Gaussian mixture distribution as $ \lambda_1\mathcal{N}(\mu_1, \Sigma_1) \oplus \ldots \oplus \lambda_m \mathcal{N}(\mu_m, \Sigma_m) $. In order to learn a Gaussian mixture model, one needs to obtain the mean vectors and the covariance matrices of every Gaussian which contributes to the mixture, and usually also the mixing weights. The number $ m $ is called the \emph{rank} of the mixture representation. 

However, learning a Gaussian mixture model is not necessarily a well-posed task, since Gaussian mixture models are not \emph{identifiable} in the strict sense: A statistical model is called {identifiable}, if two different choices of model parameters produce distinct model distributions. Trivial reasons prevent Gaussian mixture models from being identifiable. E.g., it is possible to permute the parameters, to choose two Gaussians with identical parameters (vs. adjusting the weights), to add Gaussians with mixing weight equal to zero, etc.

Ideally, we would like a result which guarantees that in ``most'' cases, it is possible to uniquely recover the parameters. Additionally, recovery should be possible from a finite amount of information about the mixture distribution. The most natural candidate for this finite amount of information would be a set of samples, drawn from the mixture distribution. However, there is a second-most natural\footnote{Informally, moments provided a coarsened perspective: From sufficiently many samples, one may clearly compute the empirical moments, which agree with the actual moments up to small error, with high probability. By directly looking at the moments, we eliminate the probabilistic aspect. As it turns out, this leads to very clean and structured results.} candidate. 
Distributions are elements of an infinite-dimensional vector space, but a canonical approach is to describe them via their \emph{moments} of degree at most $ d\in \N $. These are elements of a finite-dimensional vector space. The focus of this paper is the following question. 

\begin{question}\label{question:identifiability}
	For which ranks $ m $, dimensions $ n $ and degrees $ d $ can we recover the parameters of a general rank-$ m  $ Gaussian mixture on $ \R^n $ from its moments of degree $ d $? 
\end{question}

The term ``general'' here means generic in the sense of the Zariski topology. The results we give will only depend on the values $ m, n $ and $ d $, and not on any other properties of the mixture. In other words, for ``almost all'' rank-$ m $ degree-$ d $ Gaussian mixture moments on $ \R^n $, the answer is the same. Note that \Cref{question:identifiability} is not about computationally efficient recovery, but about information theoretic recoverability. In other words, we want to know when the parameters are uniquely determined by the moments. 

\subsection{Overview of contributions}
In this paper, we resolve most cases of degree-$ 6 $ Gaussian moment identifiability. Precisely, we show that the parameters of a generic, uniformly weighted mixture $ Y = \frac{1}{m}\mathcal{N}(\mu_1, \Sigma_1) \oplus \ldots \oplus \frac{1}{m} \mathcal{N}(\mu_m, \Sigma_m) $ of $ m $ Gaussians are uniquely determined by the moments of $ Y $ of degree $ 6 $, if $ m \le m(n) $ is at most some threshold, depending on $ n $, with asymptotic growth equal to $ m(n)=\Theta(n^4) $. 
For the same threshold $ m\leq m(n) $, we also show that the parameters of a Gaussian mixture $ Y =\lambda_1 \mathcal{N}(\mu_1, \Sigma_1) \oplus \ldots \oplus \lambda_m \mathcal{N}(\mu_m, \Sigma_m) $ with general parameters and general mixing weights are uniquely determined by the moments of $ Y $ of degree $ 6 $ and $ 4 $. The precise statement can be found in \Cref{thm:gmm-identifiable-deg6}. Note that for the weighted case, moments of two distinct degrees are necessary for any such result.

We complement our theoretical results with some computations of the secant dimensions of $ \GM_{d}(\C^n) $ in small numbers of variables $ n\in \{1,\ldots,19\} $, and various degrees $ d\in \{4,5,6\} $.

\subsubsection{Techniques}

Moment identifiability for Gaussian mixtures has been a long standing open problem, with progress over the last years mostly limited to the univariate case and often only proving that there exist less than infinitely many solutions (so called finite-to-one identifiability). Very recent developments allowed us to advance beyond that. Our identifiability proof uses three main ingredients:  
First, the new result of Massarenti and Mella \cite{Massarenti_Mella_2022}, see \Cref{thm:massarenti-mella}, building up on work of Casarotti and Mella \cite{Casarotti_Mella_2022}, which showed that $ m $-identifiability can, under mild conditions, be obtained from $ m+1 $ (tangential) nondefectivity. Second, a result of Nenashev \cite{nenashev2017note}, which we state in \Cref{thm:nenashev}: It resolved a large number of cases of Fröberg's 1985 conjecture on the Hilbert series of sets of forms \cite{froberg1985inequality}. And third, a carefully constructed degeneration argument, which allows to reduce secant nondefectivity to a combinatorial puzzle. 

The hidden fourth ingredient is a change in notation. 
Compared to some previous work on the moments of Gaussian mixtures, we deviate from the convention to write the $ 2d $-th moments of a Gaussian mixture as mode-symmetrization of a tensor in $ S^{2}(\C^n)^{\ot d} $ (cf. \cite{Ge_Huang_Kakade_2015},\cite{Bafna_Hsieh_Kothari_Xu_2022},\cite{Pereira_Kileel_Kolda_2022}). Instead, we interpret the moments of degree $ 6 $ of an $ n $-variate Gaussian distribution $ \mathcal{N}(\mu, \Sigma) $, with mean $ \mu\in \R^n $ and positive definite covariance matrix $ \Sigma \in \R^{n\times n} $, as the coefficients of the sextic form
\begin{align}
	s_{6}(\ell, q) = \ell^6 + 15q\ell^4 + 45q^2\ell^2 + 15q^3 \in \R[X]_{6},  
\end{align}
where $ X = (X_1,\ldots,X_n) $, $ \ell=\mu^{T}X $ and $ q = X^{T}\Sigma X $. Such a \emph{moment form} $ s_d $ does exist in each degree, and up to scaling, it is the $ d $-homogeneous part of the \emph{moment generating series}, which for a Gaussian $ \mathcal{N}(\mu, \Sigma) $ equals 
\begin{align}
	\exp\left(\ell + \frac{q}{2}\right) = \sum_{d = 0}^\infty \frac{1}{d!} \left(\ell + \frac{q}{2}\right)^{d}.
\end{align}
While the choice of notation is usually of minor importance, we wish to stress that all our results strongly rely on the fact that we can express $ s_d(\ell, q) $ as a binary polynomial in $ \ell $ and $ q $. For instance, we make use of factorizations. \Cref{table:gaussian-moments} shows the $ s_d(\ell, q) $ for various values of $ d $.

\begin{table}[h]
	\centering
	\begin{filecontents*}{data/gaussian-moments-derivatives.csv}
"degree","moments","lderivativenormalized","qderivativenormalized"
1,$\ell$,$1$,$0$
2,$\ell^2 + q$,$\ell$,$1$
3,$\ell^3 + 3q\ell$,$\ell^2 + q$,$\ell$
4,$\ell^4 + 6q\ell^2 + 3q^2$,$\ell^3 + 3q\ell$,$\ell^2 + q$
5,$\ell^5 + 10q\ell^3 + 15q^2\ell$,$\ell^4 + 6q\ell^2 + 3q^2$,$\ell^3 + 3q\ell$
6,$\ell^6 + 15q\ell^4 + 45q^2\ell^2 + 15q^3$,$\ell^5 + 10q\ell^3 + 15q^2\ell$,$\ell^4 + 6q\ell^2 + 3q^2$
7,$\ell^7 + 21q\ell^5 + 105q^2\ell^3 + 105q^3\ell$,$\ell^6 + 15q\ell^4 + 45q^2\ell^2 + 15q^3$,$\ell^5 + 10q\ell^3 + 15q^2\ell$
8,$\ell^8 + 28q\ell^6 + 210q^2\ell^4 + 420q^3\ell^2 + 105q^4$,$\ell^7 + 21q\ell^5 + 105q^2\ell^3 + 105q^3\ell$,$\ell^6 + 15q\ell^4 + 45q^2\ell^2 + 15q^3$
\end{filecontents*}
\csvreader[tabular=cc, filter={\value{csvrow}<8}, 
	table head=\hline $ d $ & Gaussian moment form of degree $ d $\\ \hline, late after line=\\]
	{data/gaussian-moments-derivatives.csv}{"degree"=\degree,"moments"=\moments}
	{\degree & \moments}
	\caption{Moment forms of a Gaussian distribution $ \mathcal{N}(\mu, \Sigma) $ in degree $ d\in \{1,\ldots, 8\} $, with $ \ell = \mu^{T}X $ and $ q = X^{T}\Sigma X $. All expressions were normalized such that the coefficient of $ \ell^{d} $ is $ 1 $. }
	\label{table:gaussian-moments}
\end{table}

\subsection{Related work}

Classically, a lot of focus had been directed towards the univariate case. Pearson's original work \cite{Pearson_1900} examined moments of a univariate rank-$ 2 $ Gaussian mixture of degree at most $ 5 $. More recent results on the univariate case are, e.g., due to Amendola, Sturmfels et al. (\cite{Amendola_Faugere_Sturmfels_2016},\cite{Amendola_Ranestad_Sturmfels_2018}), who examined univariate identifiability for higher ranks. Massarenti and Mella \cite{Massarenti_Mella_2022} recently proved identifiability for univariate Gaussian distributions, as a corollary of their new, very general identifiability theorem. 

The special case of multivariate Gaussians with \emph{identical} covariance matrices is also very well-understood: These can essentially be reduced to $ 1 $-Waring decompositions, sometimes also called symmetric tensor decompositions, see the author's doctoral thesis \cite[p. 42, p. 47]{Taveira_Blomenhofer_Thesis}. For $ 1 $-Waring decompositions, the question of identifiability is now completely understood, see \Cref{sec:ident-in-geometry}, and there are some efficient algorithms for the low-rank case, see \cite{Leurgans_Ross_Abel_1993} and \cite{Anandkumar_Ge_Hsu_Kakade_Telgarsky_2012}. 

\medskip
With the advent of big data, there emerged interest to understand Gaussian mixtures in high dimensions, preferably with full freedom on the covariance matrices. This direction was pioneered by some algorithmic recovery results: Moitra and Valiant \cite{Moitra_Valiant_2010} gave a polynomial-time algorithm to estimate the parameters of a constant-rank mixture in large dimensions. The constant-rank regime has been further improved since then, for instance under the aspect of robustness in \cite{Bakshi_Diakonikolas_Jia_Kane_Kothari_Vempala_2022} and in \cite{Liu_Moitra_2021}. Recently, there have also been attempts to give homotopy continuation based algorithms for the constant rank regime \cite{Lindberg_Amendola_Rodriguez_2023}, which seem to run well in practice, if the rank is $ 2 $. Unfortunately, the computational complexity of the homotopy continuation method is not understood. 
A major leap was the 2015 result of Ge, Huang and Kakade \cite{Ge_Huang_Kakade_2015}, who gave an algorithm to handle mixtures of rank at most $ m=\mathcal{O}(\sqrt{n}) $ on $ \R^n $. This was the first breakout from the constant-rank regime. 

The special case of \emph{centered} Gaussians, where all Gaussians have identical means, was subsequently studied from the algorithmic perspective in \cite{Garg_Kayal_Saha_2020} and \cite{Bafna_Hsieh_Kothari_Xu_2022}, and also by the author in \cite{Taveira_Blomenhofer_UPOF_2023} and \cite{Blomenhofer_Casarotti_Michalek_Oneto_2022}, the latter in joint work with Casarotti, Micha{\l}ek and Oneto. 

This paper gives, to the best of our knowledge, the first identifiability result that, asymptotically in a large number of variables $ n $, matches the rank upper bound 
\begin{align}
	m\le  \dfrac{\binom{n+5}{6}}{\binom{n+1}{2} + n} =\frac{1}{360} n^4 + \frac{1}{30} n^3 + \frac{49}{360} n^2 + \frac{13}{60} n + \frac{1}{9}
\end{align} 
from counting parameters. Our result also matches this parameter counting bound up to the correct constant in the leading $ n^4 $-term. Previous results were either limited to very low rank $ m = \Theta(\sqrt{n}) $, due to algorithmic constraints, see \cite{Ge_Huang_Kakade_2015}, or limited to the univariate case, see \cite{Amendola_Faugere_Sturmfels_2016} and \cite{Amendola_Ranestad_Sturmfels_2018}.

\subsubsection{Acknowledgements}
I wish to thank Monique Laurent for a lot of very valuable feedback. I also wish to thank my former doctoral advisor Mateusz Micha{\l}ek for the guidance I received during my time in Konstanz, and Alex Casarotti, for some very helpful discussions. Part of this work was completed while the author was supported by the Dutch Scientific Council (NWO) grant OCENW.GROOT.2019.015 (OPTIMAL).

\subsubsection{Disclosure}
The proof of skewness of tangent spaces in the case of degree $ 6 $ and $ 7 $ was part of my doctoral thesis, see  \cite[Lemma 3.3.4 and Theorem 3.3.8]{Taveira_Blomenhofer_Thesis}.

\section{Preliminaries}\label{sec:prelims}

\paragraph{Notation}

Let us write $ \N = \{1,2,3,\ldots \} $ for the set of natural numbers, with $ \N_0 = \N \cup \{0\} $. We work with polynomials in real or complex coefficients. For $ K\in \{\R, \C\} $, the polynomial ring in $ n $ variables $ X = (X_1,\ldots,X_n) $ is denoted $ K[X] $. By $ K[X]_{d} $, we denote the finite-dimensional subspace of all homogeneous polynomials of degree $ d \in \N_0 $, which are also called $ d $-forms. Unless explicitly stated otherwise, a polynomial on $ \C^n $ is by default assumed to be in variables $ X $. Both algebraic unknowns and random vectors are always denoted with capital letters. 

We assume familiarity with the notion of a Gaussian distribution $ \mathcal{N}(\mu, \Sigma) $, which is parametrized by a mean vector $ \mu \in \R^n $ and a positive definite covariance matrix $ \Sigma \in \R^{n\times n} $. For each $ \alpha \in \N_0^n $ and each random vector $ Y $ with distribution $ Y\sim \mathcal{N}(\mu, \Sigma) $, the expected value 
\begin{align}
	\E[Y^{\alpha}] = \int_{\R^n} x^{\alpha} d\mathcal{N}(\mu, \Sigma)(x) 
\end{align}
of the monomial $ Y^{\alpha} $ exists. The value $ |\alpha| = \alpha_1 + \ldots + \alpha_n $ is called the \emph{degree} or \emph{order} of the moment. If one knows all moments of order $ d $ of some probability distribution, then one may compute the expectation $ \E[f(Y)] $ of any given $ d $-form $ f\in \C[X]_{d} $. Moments are related to the coefficients of the \emph{moment generating series}
\begin{align}
	\mathcal{M}(Y) := \sum_{\alpha \in \N_0^n} \frac{1}{|\alpha|!} \binom{d}{\alpha} \E[Y^{\alpha}] X^{\alpha} \in \R[[X]],
\end{align}
which is a formal power series in $ X $. In \Cref{sec:gaussian-moments}, we recall combinatorial expressions for the homogeneous parts of $ \mathcal{M}(Y) $ in terms of the mean and the covariance. 

Gaussian mixtures are real objects by nature, but it helps a lot to examine identifiability over the complex field. Therefore, we formally extend all notions related to Gaussians distributions to the domain of complex numbers. In particular, the \emph{Gaussian moment variety}, introduced in \Cref{def:gm-variety}, will contain the moments of formal ``Gaussians'' $ \mathcal{N}(\mu, \Sigma) $, where $ \mu \in \C^n $ and $ \Sigma \in \C^{n\times n} $ is a symmetric matrix with complex entries. These objects do not have any statistical meaning (to the best of our knowledge), and are just considered to facilitate the analysis. 

For $ K\in \{\R, \C\} $, we endow any finite dimensional $ K $-vector space $ U \cong K^n $ with the $ K $-Zariski topology on $ U $. The closed sets with respect to this topology are the solution sets of systems of polynomial equations on $ U $. Algebraic (sub)varieties (of $ U $) are dense subsets of closed sets in $ U $. In the literature, this definition of a variety is sometimes also called a \emph{quasi-affine variety}. All our closed varieties $ V $ will be \emph{affine cones}, saying that for each $ x\in V $ and $ \lambda\in K $, also $ \lambda x \in V $. We denote the tangent space at $ x\in V $ by $ T_{x}V $. For each $ x\in V $, the tangent space is a linear subspace of $ K^n $. 

We use the classic symbols from Landau notation, $ f=\mathcal{O}(g) $ and $ f=\Theta(g) $, to denote the behaviour of certain functions in large numbers of variables $ n $.  In addition, we introduce $ f=\Theta^{\#}(g) $, to denote that $ \lim_{n\to \infty} \frac{f(n)}{g(n)} = 1 $. This relation is stricter than $ f=\Theta(g) $: While $ f=\Theta(g) $ means that the functions $ f $ and $ g $ asymptotically differ only by a multiplicative constant, $ f=\Theta^{\#}(g) $ means that this constant is one. 

The base case of the proof of \Cref{prop:gm6-not-1-tangentially-weakly-defective} was verified on a computer. The secant dimensions presented in \Cref{sec:numerics} were also calculated numerically. Data and code for both can be found in the accompanying git repository, see \cite{Taveira_Blomenhofer_Code_Gaussian_Identifiability_2023}. When referring to specific parts of this git repository, we will use a citation with the path to the file or the folder, relative to the main folder named \texttt{gaussian-identifiability} of the repository, e.g., \cite[\texttt{path/to/folder}]{Taveira_Blomenhofer_Code_Gaussian_Identifiability_2023}. 
We used the \texttt{Julia} \cite{Julia-2017} programming language, with the non-base package \texttt{DynamicPolynomials.jl} \cite{legat2021multivariatepolynomials}. 

\subsection{Identifiability in geometry and the Waring problem}\label{sec:ident-in-geometry}

Aside from the statistical identifiability problem, there is a second notion of \emph{identifiability of secant varieties}, which stems from a geometrical perspective. Here, an algebraic variety\footnote{Precisely, an irreducible affine cone.} $ V\subseteq \C^N $ is given and the question is: When is an $ m $-fold sum of elements of $ V $ uniquely representable as an $ m $-fold sum of elements of $ V $? Just like Gaussian mixtures, secant identifiability is a problem with a similarly long history (\cite{Sylvester_1904},\cite{Te12},\cite{Hilbert_1933_letter}). The development of the theory of secant varieties was motivated by the study of the $ k $-Waring problem: Given $ m, k,d\in \N $ and $ k $-forms $ q_1,\ldots,q_m \in \C[X]_{k} $, when are these the only $ k $-forms such that the sum of their $ d $-th powers equals 
\begin{align}
	f = q_1^d + \ldots + q_m^d \text{?} 
\end{align}
The most classical case is the $ 1 $-Waring problem, which corresponds to secant identifiability of Veronese varieties. Some of the major achievements for the $ 1 $-Waring problem were the Alexander-Hirschowitz theorem \cite{hirschowitz1995polynomial}, then a series of work due to Chiantini, Ottaviani and Vannieuwenhoven on identifiability of forms of subgeneric rank (\cite{Chiantini_Ottaviani_2012},\cite{Chiantini_Ottaviani_Vannieuwenhoven_2014},\cite{Chiantini_Ottaviani_Vannieuwenhoven_2016}), and finally Galuppi and Mella's classification of all the cases where a generic form is identifiable with respect to $ 1 $-Waring decompositions \cite{galuppi2019identifiability}. 

The Waring problem for higher $ k $ is also classical. An early specific case was famously considered in 1913 by Ramanujan (\cite{ramanujan1913problem}, \cite[p. 326]{ramanujan2000collected}). Another explicit consideration of a higher order Waring problem was in 1880 due to Desboves (\cite[p. 684]{Dickson_1966}). 
Initially this problem caught the interest of mathematicians due to unexpected patterns of dependence among powers of forms. See the work of Reznick, who mentions the two classical problems from above in \cite{Reznick_2003} and \cite{Reznick_2021}. Recently, $ k $-Waring decompositions have also been studied from the viewpoint of secant identifiability, e.g., by Lundqvist, Oneto, Reznick and Shapiro \cite{Lundqvist_Oneto_Reznick_Shapiro_2019}.

\medskip 
It turns out that the statistical and the geometrical notions of identifiability are related: The Waring problem for linear forms is connected to identifiability for finitely supported distributions, and to mixtures of Gaussians with \emph{identical} covariance matrices, for the latter see the author's doctoral thesis \cite[p. 42, p. 47]{Taveira_Blomenhofer_Thesis}. The Waring problem for quadratic forms is connected to mixtures of \emph{centered} Gaussians, where every Gaussian has the same mean (\cite{Blomenhofer_Casarotti_Michalek_Oneto_2022},\cite{Taveira_Blomenhofer_UPOF_2023}). 

However, in the fully general case, Gaussian mixture identifiability is not a Waring problem. Instead, it asks about uniqueness of representations
\begin{align}
	f = s_d(\ell_1, q_1)  + \ldots + s_d(\ell_m, q_m), 
\end{align}
where $ s_d(\ell, q) $ is some explicit, bivariate (weighted homogeneous) polynomial expression in $ \ell, q $. The $ 1 $-Waring problem is recovered, if all quadratic forms $ q_i $ are set to zero. From a statistical viewpoint, setting $ q_i $ to zero degenerates the Gaussians to Dirac distributions. As a result, one obtains a problem of \emph{atom reconstruction} from finitely supported measures. The $ 2 $-Waring problem is recovered, if $ d $ is even and all linear forms $ \ell_i $ are set to zero. Therefore, Gaussian moment decompositions can be seen as sort of an ``interpolation'' between linear Waring decomposition and quadratic Waring decomposition. 

\medskip
In the following, we will briefly introduce the main technical tools and results needed from the theory of secant varieties for an abstract variety $ V $. 

\begin{defi}\label{def:identifiable}
	Let $ V $ be an irreducible affine cone in some complex vector space and $ m\in \N $. One says that $ V $ is \emph{(generically) $ m $-identifiable}, if a sum 
	\begin{align}\label{eq:def-identifiable-1}
		z = x_1 + \ldots + x_m
	\end{align}
	of $ m $ general elements of $ V $ has no decomposition as a sum of $ m $ elements of $ V $ other than \eqref{eq:def-identifiable-1}. 
\end{defi}

\begin{defi}\label{def:secant-variety}
	Let $ V $ be an irreducible affine cone in some complex vector space and $ m\in \N $. The Zariski closure of the set 
	\begin{align}
		\{x_1 + \ldots + x_m \mid x_1,\ldots,x_m\in V\}
	\end{align}
	is called the $ m $-th \emph{secant variety} of $ V $, and denoted $ \sigma_m(V) $. 
\end{defi}

The notion of identifiability exists also for individual elements of a secant variety. 

\begin{defi}
	Let $ V $ be an irreducible affine cone, $ m\in \N $ and $ t\in \sigma_m (V) $. Then, $ t $ is called $ m $-\emph{identifiable} with respect to $ V $, if there is precisely one way to write $ t $ as a sum of $ m $ elements of $ V $. 
\end{defi}

\begin{lemma}[{Terracini, \cite{Te12}}]\label{lem:terracini}
	Let $V \subseteq \C^{N}$ be an irreducible and non-degenerate affine cone of dimension $n$, $m \in \N$ and $ x_1,\ldots,x_m $ general points of $ V $. Then, for all general points $ z\in \langle x_1,\ldots,x_m \rangle $, it holds
	\begin{align}
		T_{z} \sigma_m(V) = \sum_{i = 1}^{m} T_{x_i} V
	\end{align}
	Furthermore, if the sum 
	\begin{align}
		\sum_{i = 1}^{m} T_{x_i} V = \bigoplus_{i=1}^m T_{x_i} V
	\end{align}
	is direct, then each general $ z\in \sigma_m(V) $ has only finitely many decompositions as a sum of $ m $ elements of $ V $. In that case, we say that $ V $ is \emph{$ m $-nondefective}.
\end{lemma}

We need one more technical definition, before we are ready to state the main tool that we use for identifiability. 

\begin{defi} 
	Let $ V $ an irreducible affine cone and $ x\in V $. We write 
	\begin{align}
		\Gamma_{V}(x) = \{y\in V\mid T_{y}V \subseteq T_{x} V\}.
	\end{align}
	The union of all irreducible components of $ \Gamma_{V}(x) $ that pass through $ x $ is denoted $ \mathcal{C}_{V}(x) $. We call $ \mathcal{C}_{V}(x) $ the \emph{tangential contact locus} of $ V $ at $ x $.
	
	\smallskip
	Note that the tangential contact locus always contains the line $ \C x $ of multiples of $ x $.  
	The variety $ V $ is called $ 1 $-\emph{tangentially weakly defective}, if the dimension of $ \mathcal{C}_{V}(x) $ at general $ x\in V $ is strictly greater than $ 1 $. 
\end{defi}

Tangential weak defectivity might seem like a complicated, technical constraint, but intuitively, it should be seen as a condition that the base variety we start with is ``reasonable''. For instance, if $ V \subseteq \C^N $ is a linear subspace of dimension at least $ 2 $, then $ V $ is $ 1 $-tangentially weakly defective, since all points have the same tangent space. On the other hand, it is also completely clear that a sum of $ 2 $ elements of a linear subspace $ V $ will not uniquely determine the summands. 
As a similar example, Massarenti and Mella show that the secant $ V = \sigma_2(V_0) $ of an irreducible and nondegenerate variety $ V_0 $ is never $ 1 $-tangentially weakly defective, and also never $ 2 $-identifiable \cite[Proposition 2.20 and Remark 3.11]{Massarenti_Mella_2022}. The reason is that by Terracini's Lemma, for general $ x, y\in V_0 $, all general points in $ \langle x, y \rangle \subseteq \sigma_2(V_0) $ have the same tangent space.

\begin{theorem}[{Massarenti-Mella \cite[Theorem 1.5, Remark 2.3]{Massarenti_Mella_2022}}]\label{thm:massarenti-mella}
	Let $V \subseteq \C^{N}$ be an irreducible and non-degenerate affine cone of dimension $n$. Let $m \in \N$ and assume that
	\begin{enumerate}
		\item $ mn\leq N-n $ %$(m + 1)n_{p} \leq N_{p}-m$, % n_p, N_p are the projective dimensions. 
		\item $V$ is not $ 1 $-tangentially weakly defective,
		\item $V$ is not $(m + 1)$-defective.
	\end{enumerate}
	Then, $V$ is $m$-identifiable.
\end{theorem}

Massarenti and Mella formulate \Cref{thm:massarenti-mella}(2) differently. Instead, they require that $ V $ has a nondegenerate Gauss map. By \cite[Remark 2.3]{Massarenti_Mella_2022}, $ 1 $-tangential weak defectivity is equivalent to $ V $ having a degenerate \emph{Gauss map}. The Gauss map is, informally, the function that maps smooth points $ x \in V $ to their tangent spaces $ T_{x} V $, which are seen as elements of a Grassmannian variety. We will avoid talking about Gauss maps throughout this article, since there could be confusion, if too many notions are named after Carl Friedrich Gauß. 
Condition (1) also looks differently in \cite{Massarenti_Mella_2022}: Since the authors work in projective space, both $ n $ and $ N $ are shifted by one, and the terms are arranged differently.

\subsection{Gaussian Moments}\label{sec:gaussian-moments}

A large amount of information about a distribution is stored in its moment forms.\footnote{Assuming the distribution is sufficiently nice, in the sense that integrals of polynomials exist.}  
The \emph{moment forms} of a Gaussian distribution with mean $ \mu \in \R^n $ and positive definite covariance matrix $ \Sigma \in \R^{n\times n}$ are the homogeneous parts in $ X $ of the formal power series
\begin{align}\label{eq:moment-generating-series-1}
	\exp\left(\mu^{T}X + \frac{1}{2} X^{T}\Sigma X\right)
\end{align}
in variables $  X = (X_1,\ldots,X_n) $, which is called the \emph{moment generating series} of $ \mathcal{N}(\mu, \Sigma) $. Writing $ \ell := \mu^{T}X $ and $ q := X^{T}\Sigma X $, we can conveniently expand this representation 
\begin{align}\label{eq:moment-generating-series-2}
	\exp\left(\ell + \frac{q}{2}\right) = \sum_{d = 0}^{\infty} \frac{1}{d!} \left(\ell + \frac{q}{2}\right)^d
\end{align}
Since $ \ell $ and $ q $ have different degrees, it takes a bit of effort to sort the representation \eqref{eq:moment-generating-series-2} into $ d $-homogeneous parts, but the result for the $ d $-homogeneous part is 
\begin{align}\label{eq:moment-form}
	\frac{1}{d!} \sum_{k=0}^{\lfloor d/2\rfloor} 2^{-k}\duonomial{d}{k} q^{k}\ell^{d-2k}.
\end{align}
Here, the combinatorial expression 
\begin{align}
	\duonomial{d}{k} := \dfrac{d!}{k!(d-2k)!}  = \frac{d!\binom{d-k}{k}}{(d-k)!}  = k!\binom{d}{k}\binom{d-k}{k}
\end{align} occurs, which we will refer to as the \emph{duonomial} coefficients. The connection between statistics, moments and these combinatorial expressions is explained in detail in my doctoral thesis \cite[Chapter 3]{Taveira_Blomenhofer_Thesis}. For some concrete examples, \Cref{table:gaussian-moments} shows the first Gaussian moment forms.

\subsubsection{Gaussian moment varieties}\label{sec:gaussian-moment-varieties}
Note that the Gaussian moment forms are polynomial expressions in $ (\ell, q) $. It thus makes sense to consider the polynomial morphism which sends $ (\ell,q) $ to the degree-$ d $ moment form. This map will be easier to study over the complex numbers, but we will give an argument why complex identifiability also implies identifiability over the real field (and actually all fields of characteristic zero, in analogy to \cite[Remark 3.2.5]{Taveira_Blomenhofer_Thesis}). 

\begin{defi}\label{def:gm-variety}
	The degree-$ d $ \emph{Gaussian moment variety} $ \GM_{d}(\C^n) $ is the (Zariski) closure of the image of the map 
	\begin{align}\label{eq:moment-map-1}
		s_d\colon (\ell, q) \mapsto \sum_{k=0}^{\lfloor d/2\rfloor} 2^{-k}\duonomial{d}{k} q^{k}\ell^{d-2k}
	\end{align}
	For $ d\ge 4 $, it is an irreducible, nondegenerate variety in $ \C[X]_{d} $ of dimension $ \binom{n+1}{2} + n = \frac{1}{2}n(n+3) $. 
	Formally, 
	\begin{align}
		\GM_{d}(\C^n) = \overline{\{s_d(\ell, q) \mid \ell \in \lf(\C^n), q\in \quadf(\C^n) \}} \subseteq \C[X]_{d}
	\end{align}
\end{defi}
\begin{proof}[Remarks towards \Cref{def:gm-variety}]
	Note that $ \GM_{d}(\C^n) $ is irreducible, since it is the closure of the image of a polynomial map. The variety is nondegenerate, i.e., not contained in a proper subspace of $ \C[X]_{d} $, since it contains the (nondegenerate) Veronese variety of $ d $-th powers of linear forms, as the image of all expressions $ s_d(\ell, 0) $. 
\end{proof}

The tangent space at a general point of the Gaussian moment variety can be found by deriving curves $t \mapsto s_{d}(\ell + th, q+tp) $. This yields the following: 

\begin{prop}\label{prop:gmm-parametrization-tangent}
	For $ d\ge 3 $ and general $ (\ell, q) $, the tangent space of $ \GM_{d}(\C^n) $ at $ s_{d}(\ell, q) $ is the set of expressions
	\begin{align}
		&s_{d-1}(\ell, q) h + s_{d-2}(\ell, q) p \nonumber \\
		\subjto \qquad &h\in \C[X]_{1}, p\in \C[X]_{2}
	\end{align}
\end{prop}
\begin{proof}
	Note that by definition, $ s_d(\ell, q) $ is the degree-$ d $ homogeneous part of the moment generating series $ \mathcal{M}(\ell, q) := \exp(\ell + \frac{q}{2}) $. When deriving $ s_d $, we can therefore make use of the fact that the exponential series behaves nice under derivatives. For instance, $ \frac{d}{dt} \mathcal{M}(\ell + th, q)|_{t=0} = h \exp(th + \ell + \frac{q}{2})|_{t=0} = h \mathcal{M}(\ell, q)$. Hence, the derivative along a curve through $ \ell $ shifts all homogeneous parts by one degree, so the $ \ell $ derivative of $ s_d $ equals $ s_{d-1} $, up to a normalization  constant. Similarly, we see that $ \frac{d}{dt} \mathcal{M}(\ell, q+tp)|_{t=0} = \frac{p}{2}\exp(\ell + \frac{q}{2}) $. Therefore, up to a constant, the $ q $-derivative of $ s_{d} $ equals $ s_{d-2} $.  
\end{proof}

\subsection{Nenashev's theorem and Fröberg's conjecture}

The interaction of tangent spaces at various general points $ s_{d}(\ell_i, q_i) $ plays a crucial role in identifiability, as seen, e.g., from Terracini's \Cref{lem:terracini} and \Cref{thm:massarenti-mella}. While for \emph{fixed} $ m, d $ and $ n $, determining the dimension of the sum of $ m $ general tangent spaces to $ \GM_{d}(\C^n) $ is a problem of Linear Algebra, unfortunately, it is often quite difficult to determine these dimensions as functions of $ m, d $ and $ n $. 

In earlier work on some special cases of Gaussian mixtures \cite{Blomenhofer_Casarotti_Michalek_Oneto_2022}, it was possible to describe such a sum of tangent spaces as the graded component of an ideal generated by powers of forms. Then, knowledge about the Hilbert series of ``general'' ideals could be used. One major tool here was Nenashev's partial resolution of Fröberg's conjecture. 

\begin{theorem}[{Nenashev, \cite[Theorem 1]{nenashev2017note}}]\label{thm:nenashev}
	Let $ a\in \N $ and let $ \mathcal{D} $ be a variety of degree-$ a $ forms which is closed under the canonical action of $ \GL_{n}(\C) $ on forms. For $ m, h\in \N $, general $ g_1,\ldots, g_m \in \mathcal{D} $, and $ I=(g_1,\ldots, g_m) $, as long as 
	\begin{align}
		m&\le \dfrac{\dim S^{a+h}(\C^n)}{\dim S^{h}(\C^n)} - \dim S^{h}(\C^n) \qquad \text{or} \label{eq:nenashev-1}\\
		m&\ge \dfrac{\dim S^{a+h}(\C^n)}{\dim S^{h}(\C^n)} + \dim S^{h}(\C^n)\label{eq:nenashev-2}
	\end{align}
	it holds that: 
	\begin{align}
		&(g_1,\ldots, g_m)_{a+h} = (g_1)_{a+h} \oplus \ldots \oplus (g_m)_{a+h}\text{ for $m$ as in \ref{eq:nenashev-1}}.\label{eq:nenashev-hf-1} \\
		&(g_1,\ldots, g_m)_{a+h} = S^{a+h}(\C^n) \qquad \qquad \qquad\,\text{ for $m$ as in \ref{eq:nenashev-2}}.\label{eq:nenashev-hf-2}
	\end{align}
\end{theorem}

A special case, classically considered by Alexander and Hirschowitz, is when the variety $ \mathcal{D} $ in \Cref{thm:nenashev} is the Veronese variety of powers of linear forms. 

\begin{theorem}[{Alexander-Hirschowitz, cf. \cite[Theorem 1.1]{Brambilla_Ottaviani_2008}, originally in \cite{hirschowitz1995polynomial}}] \label{thm:alexander-hirschowitz}
	Let $ n, d, m\in \N_{\ge 2} $ and $ I = (\ell_{1}^{d-1},\ldots,\ell_{m}^{d-1}) $ be an ideal generated by $ (d-1) $-th powers of \emph{generic} linear forms $ \ell_{1},\ldots,\ell_{m} $ on $ \C^n $. Then, the homogeneous component $ I_{d} $ of $ I $ has the expected dimension, which is $ \min \{mn, \binom{n+d-1}{d}\} $, in all but the following exceptional cases:
	\begin{itemize}
		\item $ d = 2 $ and $ m = 2,\ldots,n-1 $.
		\item $ n = 3 $, $ d = 4 $, $ m = 5 $.
		\item $ n = 4 $, $ d = 4 $, $ m = 9 $. 
		\item $ n = 5 $, $ d = 3 $, $ m = 7 $.
		\item $ n = 5 $, $ d = 4 $, $ m = 14 $.
	\end{itemize}
\end{theorem}

Note that we rewrote the result to fit our notation: Brambilla and Ottaviani \cite{Brambilla_Ottaviani_2008}, just like Alexander and Hirschowitz \cite{hirschowitz1995polynomial}, formulate \Cref{thm:alexander-hirschowitz} from a dual perspective. Also, in \cite{Brambilla_Ottaviani_2008}, they use projective dimensions.

Unfortunately, the tangent spaces to secants of $ \GM_{d}(\C^n) $ do not have a structure simple enough to be covered by \Cref{thm:nenashev}. However, via some careful degeneration argument, we will still be able to make use of \Cref{thm:nenashev}. The following observation is trivial, but we will need to refer to it so often, that it is worth stating it as a Lemma.

\begin{lemma}\label{lem:disjoint-variable-skewness}
	Let $ m\in \N $ and let $ f_1,\ldots,f_m \in \C[X] $ be forms of degree $ d\in \N $. Let $ k_1\in \N $ be such that $ (f_1,\ldots,f_m)_{d+k_1} = (f_1)_{d+k_1} \oplus \ldots \oplus (f_m)_{d+k_1} $ in $ \C[X] $. Let $ Y = (Y_1,\ldots,Y_{n'}) $ another vector of variables. Then for all $ k_2\in \N $ and all bihomogeneous $ h_1,\ldots,h_m\in \C[X, Y]_{k_1+k_2} $ of $ X $-degree $ k_1 $ such that 
	\begin{align}
		0 = \sum_{i = 1}^{m} f_i h_i, 
	\end{align}
	it holds $ h_1 = \ldots = h_m = 0 $. 
\end{lemma}
\begin{proof} 
	Plug in an arbitrary value $ y\in \C^{n'} $ and observe that $ 0 = \sum_{i = 1}^{m} f_i h_i(X, y) $. Since $ f_1,\ldots,f_m $ are linearly independent, it follows $ h_1(X, y) = \ldots = h_m(X, y) = 0 $ for all $ y\in \C^n $. This shows the claim. 
\end{proof}

\subsection{Non-Tangential weak defectivity of Gaussian moment varieties}

Tangential weak nondefectivity is a technical requirement of Massarenti-Mellas identifiability theorem. Unlike the condition on tangent spaces, note that $ 1 $-tangential nondefectivity is a property of just the base variety $ \GM_{d}(\C^n) $, not of the secants.  
We will give a proof technique to show that for fixed $ d $, and \emph{all} $ n \in \N $, $ \GM_{d}(\C^n) $ is not $ 1 $-tangentially weakly  defective. For the scope of this paper, only the cases $ d = 5,6,7 $ are relevant. The idea is to use a degeneration argument due to Chiantini and Ottaviani \cite{Chiantini_Ottaviani_2012}, and then show that it allows to verify the statement for all $ n $, by verifying it for $ n = 2 $. The last step can be done on a computer, for fixed $ d $. 
Therefore, this section is mostly considered with the reduction to $ n = 2 $. For the reduction step, we rely on factorizations of the bivariate polynomial expression $ s_{d}(\ell, q) $. We believe that our proof technique extends to higher values of $ d $ and all values of $ n $. However, since our proof involves a machine computation, it can only prove the statement for finitely many values of $ d $ (and all $ n $).

\begin{lemma}\label{lem:sd-no-common-quadfactors} Let $ n\ge 2 $, $ d\in \{4,\ldots,9\} $ and let $ \ell \in \C[X]_{1}, q \in \C[X]_{2}$ be such that 
	\begin{align}\label{eq:sd-1-and-sd-2-vanish}
		s_{d-1}(\ell, q) = 0 = s_{d-2}(\ell, q).
	\end{align}
	Then $ \ell = 0 = q $. 
\end{lemma}
\begin{proof}
	For each $ k\in \N $, we may view $ s_k(L, Q) \in \Q[L, Q] $ as a bivariate polynomial in variables $ (L, Q) $. For even $ k $, it is simultaneously also a bivariate form in variables $ (L^2, Q) $. As the latter, it splits into a product of linear factors, or better called ``quadratic factors'', of the form $ Q + cL^2 $, with nonzero $ c \in \R $. If $ k $ is odd, then $ \frac{s_k(L, Q)}{L} $ is a bivariate form in $ (L^2, Q) $. Thus, $ s_k(L, Q) $ splits into $ L $ times a product of such quadratic factors, if $ k $ is odd. Let us write $ \hat{s}_k(L, Q) $ to denote $ s_k(L, Q)/L $, if $ k $ is odd, and simply $ \hat{s}_k := s_k $, if $ k $ is even.

	Next, assume $ q, \ell $ are such that $ s_{d-1}(\ell, q) = 0 = s_{d-2}(\ell, q) $. If either of $ \ell $ or $ q $ is zero, we readily see that the other one must be zero, too, by looking at some vanishing quadratic factor. Therefore, assume to the contrary that $ \ell \ne 0 \ne q $. Since $ (d-1) s_{d-2} = \partial_{L} s_{d-1} $, \Cref{eq:sd-1-and-sd-2-vanish} means that for each value $ \gamma \in \C $, the univariate polynomial $ \hat{s}_{d-1}(L, \gamma) $ has a double root. Let us choose $ \gamma = 1 $, so that $ s_{d-1}(L, 1) \in \Z[L] $ is a polynomial with integer coefficients.  
	
	Now, we see from Eisenstein's criterion that $ \hat{s}_{d-1}(L, 1) $ is irreducible over $ \Q $, where we apply the criterion with the primes $ 3,3,5,5,7,7 $ for $ d-1 = 3,4,5,6,7,8 $. Confer \Cref{table:gaussian-moments} for the coefficients of $ s_{d-1} $. In particular, $ \hat{s}_{d-1}(L, 1) $ is separable, and hence has no double roots over the real or complex numbers. As a concrete example, for $ d = 6 $, the factorizations are
	\begin{align}
		s_{5}(L, Q) &= 15(Q + (\sqrt{10}+5)L^2) (Q + (\sqrt{10}-5)L^2)L\\
		s_{4}(L, Q) &= 3 (Q + (\sqrt{6}+3)L^2) (Q + (\sqrt{6}-3)L^2)
	\end{align}
	From that, one easily sees that there is no common root. 
\end{proof}

We are now ready to verify Condition (2) from \Cref{thm:massarenti-mella}. While our focus is on the case of degree $ 6 $, it does not hurt to verify it for a few other degrees, too. 

\begin{prop}\label{prop:gm6-not-1-tangentially-weakly-defective}
	$ \GM_{d}(\C^n) $ is not $ 1 $-tangentially weakly defective for $ d = 5,\ldots,8 $.
\end{prop}
\begin{proof}
	The claim is that for general $ (\ell, q) $ the tangential contact locus $ \mathcal{C}(s_d(\ell, q)) $ has projective dimension $ 0 $ locally at $ s_d(\ell,q) $. In other words, we need to check the dimension of
	\begin{align}
		\Gamma(s_d(\ell, q)) = \{ s_d(\ell', p') \mid T_{s_d}(\ell', p') \subseteq T_{s_d}(\ell, q) \}
	\end{align}
	By semicontinuity, (cf. \cite[Theorem (iii)$\implies$(ii)]{Chiantini_Ottaviani_2012}), we may show that the contact locus is projectively zero-dimensional at $ (\ell, q) $ for some specific choice of $ (\ell, q) $. To this end, write the $ n $ variables as $ (X, Y, Z_{1},\ldots,Z_{n-2}) $, with $ Z = (Z_1,\ldots,Z_{n-2}) $. Choose $ Y $ as the linear form and $ X^2 $ as the quadratic form. Assume first that $ n\ge 3 $, so there is at least one $ Z $-variable. 
	
	Now, to the contrary, assume that in any Zariski open neighbourhood $ \mathcal{U} $ of $ (Y, X^2) $, there are $ (\ell, q) $, such that $ s_d(\ell, q) $ is not a multiple of $ s_d(Y, X^2) $ and such that for each pair $ (h, p) $ of a linear form $ h $ and a quadratic form $ p $, there exists a pair $ (h', p') $ such that
%	\begin{align}\label{eq:contact-locus-deg6}
%		& (\ell^5 + 10q\ell^3 + 15q^2\ell)h + (\ell^4 + 6q\ell^2 + 3q^2) p \\
%		= & (Y^5 + 10X^2Y^3 + 15X^4Y^2)h' + (Y^4 + 6X^2Y^2 + 3X^4) p'. \nonumber
%	\end{align}
	\begin{align}\label{eq:contact-locus-degd}
		&s_{d-1}(\ell, q)h + s_{d-2}(\ell,q) p \\
		=\quad &s_{d-1}(Y, X^2)h' + s_{d-2}(Y, X^2) p'. \nonumber
	\end{align}
	Split $ \ell = \ell_{X, Y} + \ell_{Z} $, $ q = q_{X,Y} + q_{(X, Y), Z} + q_{Z} $ into homogeneous parts in $ Z $, where $ q_{X, Y} $ has degree $ 0 $ in $ Z $, $ q_{(X, Y), Z} $ has degree $ 1 $ in $ Z $ and $ q_{Z} $ has degree $ 2 $ in $ Z $. Choose $ h=Z_1 $ and $ p = 0 $. The right hand side of \eqref{eq:contact-locus-degd} only has terms of degree at most $ 2 $ in $ Z = (Z_{1},\ldots,Z_{n-2}) $. Thus, on the left hand side, all terms of degree at least $ 3 $ in $ Z $ must vanish.  Since the $ s_k $ are bihomogeneous maps for each $ k\in \N $, the part of $ s_k(\ell, q) $ of highest $ Z $-degree is $ s_k(\ell_{Z}, q_{Z}) $. We thus obtain, by looking at the part of degree $ d \ge 3 $ in $ Z $, that 
	\begin{align}\label{eq:sd-1-is-zero-in-Z}
		s_{d-1}(\ell_{Z}, q_{Z}) Z_1 = 0.
	\end{align}
	Now, we choose $ p=Z_1^2 $ and $ h = 0 $. Again, we get an identity just like \Cref{eq:contact-locus-degd}, but with potentially different $ h' $ and $ p' $. In particular, we can look at the part of degree $ d $ in $ Z $ on both sides, which this time yields
	\begin{align}\label{eq:sd-2-is-zero-in-Z}
		s_{d-2}(\ell_{Z}, p_{Z}) Z_1^2 = 0.
	\end{align}
	By \Cref{lem:sd-no-common-quadfactors}, Equations \eqref{eq:sd-1-is-zero-in-Z} and \eqref{eq:sd-2-is-zero-in-Z} together imply that $ \ell_{Z} = 0 = q_{Z} $. 
	The mixed term $ q_{(X, Y), Z} $ of degree $ 1 $ in $ Z $ is then the only term left which depends on $ Z $. Choosing again $ h = 0 $ and $ p = Z_1^2 $ yields, by looking at the part of largest degree in $ Z $, that
	\begin{align}
		q_{(X, Y), Z}^{\lfloor \frac{d}{2} \rfloor} Z_1^2  = 0
	\end{align} 
	Hence, also $ q_{(X, Y), Z} = 0 $. We conclude that $ \ell $ and $ q $ only depend on the two variables $ X $ and $ Y $. Thus, it suffices to verify the claim for $ n = 2 $ variables on a computer, see the code in the accompanying git repository \cite[\texttt{tangential-contact-locus}]{Taveira_Blomenhofer_Code_Gaussian_Identifiability_2023}. Note that we need the assumption $ d\ge 5 $ for the base case verification. 
\end{proof}

Note that in the above proof, showing irreducibility of $ \hat{s}_k(L, 1) $ with Eisenstein's criterion is convenient, but unfortunately only works up to $ k = 8 $.

\subsection{Moment identifiability of Gaussian mixtures}

\begin{notation}\label{not:moment-form}
	For a Gaussian mixture $ Y = \lambda_1\mathcal{N}(\mu_1, \Sigma_1) \oplus \ldots \oplus \lambda_m \mathcal{N}(\mu_m, \Sigma_m) $, we write $ \mathcal{M}_d(Y) $ for the form of degree-$ d $ moments of $ Y $. Formally, with the notation of \Cref{def:gm-variety}, it holds that
	\begin{align}
		\mathcal{M}_d(Y) = \sum_{i = 1}^{m} \lambda_i s_d(X^{T}\mu_i, X^{T}\Sigma_iX) \in \R[X]_{d}
	\end{align}
\end{notation}

\begin{defi}\label{def:moment-identifiable}
	A mixture $ Y = \lambda_1\mathcal{N}(\mu_1, \Sigma_1) \oplus \ldots \oplus \lambda_m \mathcal{N}(\mu_m, \Sigma_m) $ of $ m \in \N $ Gaussians is called \emph{moment identifiable} in degree $ d $, if $ \mathcal{M}_d(Y) $ is an $ m $-identifiable element of $ \GM_{d}(\C^n) $. 
\end{defi}

\subsection{Non-uniformly weighted mixtures}

From moments of one, fixed degree, it is never possible to identify the means, covariances and mixing weights of a Gaussian mixture $ Y = \lambda_1\mathcal{N}(\mu_1, \Sigma_1) \oplus \ldots \oplus \lambda_m \mathcal{N}(\mu_m, \Sigma_m) $ altogether. This is because  $ Y $ and 
\begin{align}
	Y' := \frac{1}{m} \mathcal{N}(\sqrt[d]{m\lambda_1}\mu_1, \sqrt[d/2]{m\lambda_1}\Sigma_1) \oplus \ldots \oplus \frac{1}{m} \mathcal{N}(\sqrt[d]{m\lambda_m} \mu_m, \sqrt[d/2]{m\lambda_m}\Sigma_m) 
\end{align} 
have the same moments of degree-$ d $, due to bihomogeneity of the parametrization. If one wants to prove uniqueness of both the parameters and the mixing weights, it is therefore necessary to use moments of at least two different degrees. In this section, we will make the statement formal that ``uniformly weighted identifiability from degree-$ d $ moments implies weighted identifiability from moments of degrees $ d $ and $ d-2 $. In particular, this implies that one does not have to care about weights, unless the mixture is very close to the generic rank.

\begin{lemma}\label{lem:reduce-weighted-to-uniform}
	Let $ m, n\in \N $ such that $ \GM_{d}(\C^n) $ is $ m $-identifiable. %and $ d\ge 5 $. ???
	Let $ Y = \lambda_1\mathcal{N}(\mu_1, \Sigma_1) \oplus \ldots \oplus \lambda_m \mathcal{N}(\mu_m, \Sigma_m) $ a (weighted) mixture of $ m $ Gaussians, with general parameters and mixing weights. Let $ Z = \rho_1 \mathcal{N}(\nu_1, T_1) \oplus \ldots \oplus \rho_m \mathcal{N}(\nu_m, T_m) $, another mixture of $ m $ Gaussians such that 
	\begin{align}\label{eq:moments-agree-of-Y-and-Z}
		\mathcal{M}_d(Y) = \mathcal{M}_d(Z), \text{ and } \mathcal{M}_{d-2}(Y) = \mathcal{M}_{d-2}(Z). 
	\end{align}
	Then $ Y = Z $, and, up to permutation, $ \lambda_i = \rho_i, \mu_i = \nu_i $ and $ \Sigma_i = T_i $ for all $ i\in \{1,\ldots,m\} $. 
\end{lemma}
\begin{proof}
	From \Cref{eq:moments-agree-of-Y-and-Z}, we obtain that 
	\begin{align}
		\sum_{i = 1}^{m} \lambda_i s_d(\mu_i, \Sigma_i) = \sum_{i = 1}^{m} \rho_i s_d(\nu_i, T_i)
	\end{align}
	Since $ s_d $ is bihomogeneous, we can pull the weights in and see that 
	\begin{align}\label{eq:gmm-identity-weights-pulled-in}
		\sum_{i = 1}^{m} s_d(\sqrt[d]{\lambda_i}\mu_i, \sqrt[d/2]{\lambda_i}\Sigma_i) = \sum_{i = 1}^{m} s_d(\sqrt[d]{\rho_i}\nu_i, \sqrt[d/2]{\rho_i}T_i)
	\end{align}
	Due to generality of $ \lambda_i, \mu_i $ and $ \Sigma_i $, the left-hand side of \Cref{eq:gmm-identity-weights-pulled-in} is a general element of $ \sigma_m\GM_{d}(\C^n) $, and therefore has a unique Gaussian moment decomposition by assumption. Thus, we conclude that 
	\begin{align}\label{eq:reduce-weighted-to-uniform-3}
		\{(\sqrt[d]{\lambda_i}\mu_i, \sqrt[d/2]{\lambda_i}\Sigma_i) \mid i=1,\ldots,m\} = \{(\sqrt[d]{\rho_i}\nu_i, \sqrt[d/2]{\rho_i}T_i) \mid i=1,\ldots,m\}.  
	\end{align}
	In particular, all $ \rho_i $ are nonzero. Wlog, renumber the parameters and weights such that $ (\nu_i, T_i) = (\alpha_i \mu_i, \alpha_i^2 \Sigma_i) $, where $ \alpha_i = \sqrt[d]{\frac{\lambda_i}{\rho_i}} $. 
	By \Cref{eq:moments-agree-of-Y-and-Z}, we have a similar identity between moments of degree $ d-2 $:
	\begin{align}
		\sum_{i = 1}^{m} \lambda_i s_{d-2}(\mu_i, \Sigma_i) = \sum_{i = 1}^{m} \rho_i s_{d-2}(\nu_i, T_i)
	\end{align}
	Plugging in what we know for $ (\nu_i, T_i) $, we obtain 
	\begin{align}
		0 = \sum_{i = 1}^{m} (\lambda_i - \rho_i \alpha_i^{d-2}) s_{d-2}(\mu_i, \Sigma_i).
	\end{align}
	The forms $ s_{d-2}(\mu_i, \Sigma_i) $ are linearly independent for $ i=1,\ldots,m $. (Indeed, as otherwise identifiability in degree $ 6 $ could not hold, due to the nested nature of tangent spaces). We conclude that 
	\begin{align}
		\forall i=1,\ldots,m\colon \qquad \frac{\lambda_i}{\rho_i} = \left(\frac{\lambda_i}{\rho_i}\right)^{\frac{d-2}{d}} 
	\end{align}
	This is only possible, if $ \lambda_i = \rho_i $ for all $ i\in \{1,\ldots,m\} $. Since the weights are equal, it readily follows from \Cref{eq:reduce-weighted-to-uniform-3} that the means and covariances must be equal, too.
\end{proof}

\section{Gaussian Mixture Moment Identifiability}

After the preparation work, let us come towards the main result of this paper, which concerns the case of degree $ 6 $ identifiability. 

\begin{reminder}
	The degree-$ 6 $ \emph{Gaussian moment variety} $ \GM_{6}(\C^n) $ is the closure of 
	\begin{align}
		\{\ell^6 + 15q\ell^4 + 45q^2\ell^2 + 15q^3 \mid \ell \in \C[X]_{1}, q\in \C[X]_{2}\}, 
	\end{align}
	which is the image of the map $ s_{6} $ from \Cref{def:gm-variety}. It is an irreducible, nondegenerate subvariety of $ \C[X]_{6} $ of dimension $ \binom{n+1}{2} + n = \frac{1}{2} n(n+3) $.
\end{reminder}

\begin{reminder}\label{prop:gmm-parametrization-deg6}  % [{Cf. \Cref{prop:gmm-parametrization-tangent}}]
	For general $ (\ell, q) $, the tangent space at $ s_{6}(\ell, q) $ equals
	\begin{align}
		\{(\ell^5 + 10q\ell^3 + 15q^2\ell)h + (\ell^4 + 6q\ell^2 + 3q^2) p \mid h\in \C[X]_{1}, p\in \C[X]_{2}\} \nonumber
	\end{align}
\end{reminder}
\begin{proof}
	\Cref{prop:gmm-parametrization-tangent} shows that the image of the tangent map $ T_{\ell, q}s $ consists of the expressions $ s_{5}(\ell, q)h + s_{4}(\ell, q)p  $. Combine with \Cref{table:gaussian-moments} to obtain the explicit expression above. 
\end{proof}

\begin{lemma}\label{lem:gm-deg6-skewtangents}
	Let $ n\in \N $. There exists some function $ m = m(n)=\Theta(n^4) $ with the following property: If $ \ell_{1},\ldots,\ell_{m}$ are general linear forms and $ q_1,\ldots,q_m $ are general quadratic forms on $ \C^n $, then the tangent spaces $ T_{\ell_{i}, q_i} \GM_{6}(\C^n)$, where $ i\in \{1,\ldots, m\} $, are skew spaces. 
\end{lemma}
\begin{proof} 
	We will show that $ \sum_{i = 1}^{m} \im T_{s_{6}}(\ell_{i}, q_i) $ has the maximum possible (a.k.a.  expected) dimension $m (\binom{n+1}{2} + n) $. To show the claim for general parameter choices, it is sufficient to find a specialized choice of $ (\ell_i, q_i) $, for which
	\begin{align}\label{eq:gm6-tangent-expression}
		0 = \sum_{i = 1}^{m} (\ell_i^5 + 10q_i\ell_i^3 + 15q_i^2\ell_i) h_i  +  \sum_{i = 1}^{m}  (\ell_i^4 + 6q_i\ell_i^2 + 3 q_i^2)p_i
	\end{align}
	implies $ h_1 = \ldots = h_m = 0 = p_1 = \ldots = p_m = 0 $, whenever the $ h_i $ are linear forms and the $ p_i $ are quadratic forms on $ \C^n $. 
	
	We produce our specialized choice with a \emph{variable splitting trick}: Rewrite the variables as $ (X, Y) = (X_1,\ldots,X_{n_1}, Y_{1},\ldots,Y_{n_2}) $, with $ n_1+n_2 = n $, and assume that all $ q_i\in \R[X]_{2} $ are generic quadratic forms in $ X $, while $ \ell_i\in \R[Y]_{1} $ are generic linear forms in $ Y $. Then, \Cref{eq:gm6-tangent-expression} splits into a system of $ 7 $ equations, which correspond to the parts of degree $ 0,\ldots,6 $ in $ X $. At the same time, the $ h_i = h_{i, X} + h_{i, Y} $ split into an $ X $-part and a $ Y $-part, while the $ p_i = p_{i, X} + p_{i, X, Y} + p_{i, Y} $ split into a pure $ X $-part $ p_{i, X} $, a pure $ Y $-part $ p_{i, Y} $, and a part bilinear in $ (X, Y) $. We start by looking at the part of degree $ 6 $ in $ X $. Here, only one term can contribute, so we get
	\begin{align}\label{eq:gmm-ts-deg6-X}
		0 = 3\sum_{i=1}^m q_i^2p_{i, X}
	\end{align}
	Nenashev's result, \Cref{thm:nenashev}, guarantees that 
	\begin{align}\label{eq:gm-deg6-skew-ideals}
		(q_1^2,\ldots,q_m^2)_{6} = (q_1^2)_{6} \oplus \ldots \oplus (q_m^2)_{6} 
	\end{align}
	as long as $ m=\Theta(n_1^4) $. It follows that $ p_{i, X} = 0 $ for all $ i $. Next, we look at the terms whose $ X $-degrees are $ 4 $. We obtain the equation:
	\begin{align}
		0 = 15\sum_{i=1}^m q_i^2\ell_i h_{i, Y} + 3 \sum_{i=1}^m q_i^2 p_{i, Y}
	\end{align}
	Note that the term $ \sum_{i=1}^m q_i\ell_{i}^2 p_{i, X} $ cannot make a contribution since we just showed that all $ p_{i, X}$ vanish. \Cref{eq:gm-deg6-skew-ideals} implies that $ q_1^2,\ldots, q_m^2 $ are linearly independent. Thus, \Cref{lem:disjoint-variable-skewness} implies that $ -5\ell_i h_{i, Y}  = p_{i, Y} $ for all $ i $ (choose $ (k_1,k_2) = (0, 2) $ in the Lemma). Let us plug this newly-obtained identity into the part of $ Y $-degree $ 6 $ from \eqref{eq:gm6-tangent-expression}, which is 
	\begin{align}
		0 = \sum_{i=1}^m \ell_i^5h_{i, Y} + \sum_{i=1}^m \ell_i^4 p_{i, Y}
	\end{align}
	to obtain 
	\begin{align}\label{eq:gm-deg6-skewtangents-5}
		0 = \sum_{i=1}^m \ell_i^5h_{i, Y} - 5\sum_{i=1}^m \ell_i^5 h_{i, Y} = -4\sum_{i=1}^m \ell_i^5h_{i, Y}
	\end{align}
	From the Alexander-Hirschowitz Theorem, see \Cref{thm:alexander-hirschowitz}, we obtain that the degree-$ 6 $ component of the ideal $ (\ell_1^5,\ldots,\ell_{m}^5)_{6} = (\ell_1^5)_{6} \oplus \ldots \oplus (\ell_{m}^5)_{6}$ in $ \C[Y] $ is a direct sum, as long as $ mn_1 $ is at most the dimension of $ \C[Y]_{6} $. Precisely, this is satisfied for all $ m\leq \frac{1}{n_2}\binom{n_2+5}{6} = \Theta(n_2^5) $. Thus, we conclude that $ h_{i, Y} = 0 $ for all $ i $ (and thus, also $ p_{i, Y} = -4\ell_i h_{i, Y} = 0 $). Let us repeat the same procedure with the part of degree $ 5 $ in $ X $, which is 
	\begin{align}
		0 = 15 \sum_{i=1}^m q_i^2 \ell_{i} h_{i, X} + 3\sum_{i=1}^m q_i^2 p_{i, X, Y},
	\end{align}
	This readily yields that $ p_{i, X, Y} = -5\ell_{i} h_{i, X} $ for all $ i $, again by applying \Cref{lem:disjoint-variable-skewness}, with $ k_1=k_2=1 $ therein. We can plug that into
	\begin{align}
		0 = \sum_{i=1}^m \ell_{i}^5 h_{i, X} + \sum_{i=1}^m \ell_{i}^4  p_{i, X, Y},
	\end{align}
	which is the $ X $-degree-$ 1 $ part of \eqref{eq:gm6-tangent-expression}, to obtain
	\begin{align}\label{eq:gm-deg6-skewtangents-8}
		0 = -4\sum_{i=1}^m \ell_{i}^5 h_{i, X} 
	\end{align}
	\Cref{lem:disjoint-variable-skewness} yields $ h_{i, X} = 0 $ and thus $ h_{i} = 0 $ for all $ i $. As we have $ p_{i, X, Y} = -4\ell_{i} h_{i, X} = 0 $ for all $ i $, we obtain $ p_{1} = \ldots = p_{m} = 0 $, since we showed that all $ X $-homogeneous components of the $ p_i $ vanish. This proves that a sum of $ m $ general tangent spaces to $ \GM_{6}(\C^n) $ is direct, as long as $ m<\frac{2}{n_1(n_1+3)}\binom{n_1+5}{6} - \frac{n_1(n_1+3)}{2} $ and $ m<\frac{1}{n_2}\binom{n_2+5}{6} $. Splitting the variables evenly does already yield $ m=\Theta(n^4) $, which concludes the proof. For a discussion on the optimal splitting strategy, see \Cref{rem:gm6ident-splitting-strat}. 
\end{proof}

\begin{remark}\label{rem:gm6ident-splitting-strat}
	Note that for the range of ranks $ m $ in the proof of \Cref{lem:gm-deg6-skewtangents}, the constraints involving the $ X $-variables are the bottleneck: Our proof technique requires  $ m=\mathcal{O}(n_1^4) $ and $ m=\mathcal{O}(n_2^5) $. We can therefore optimize our result, by splitting the variables unevenly: Choose $ c = c_n \in [0, 1] $ such that approximately a $ c $-fraction of the variables goes to $ Y $ and a $ (1-c) $ fraction goes to $ X $. Particularly, choosing $ c \approx \sqrt[5]{1/n} $, we obtain for large $ n $ with this trick that identifiability holds up to the correct constant, i.e., for ranks up to a function $ m_{\max}(n)=\Theta^{\#}(\dim \C[X]_{6}/\dim \GM_{6}(\C^n) ) = \frac{1}{360}\Theta^{\#}(n^4)$. 
	
	More precisely, the two constraints on $ m $ with respect to $ n_1 $ and $ n_2 $ are:
	\begin{align}
		m &\leq \dfrac{\binom{n_1 + 5}{6}}{\binom{n_1+1}{2}} - \binom{n_1+1}{2} = \frac{1}{360} n_1^4 + \frac{7}{180} n_1^3 + \frac{109}{360} n_1^2 + \frac{13}{180} n_1 + \frac{1}{3} \\
		m &\leq \dfrac{\binom{n_2 + 5}{6}}{n_2} - n_2 = \frac{1}{720} n_2^5 + \frac{1}{48} n_2^4 + 	\frac{17}{144} n_2^3 + \frac{5}{16} n_2^2 - \frac{223}{360} n_2 + \frac{1}{6}.
	\end{align}
	For any $ \varepsilon > 0 $, we can choose $ c = \sqrt[5]{1/n}^{(1-\varepsilon)} $ and plug in $ (n_1, n_2) = ((1-c)n, cn) $ (up to rounding to integers). One obtains that $ \GM_{6}(\C^n) $ is $ m $-identifiable as long as 
	\begin{align}
		m &\leq \frac{1}{360} n^4 + \mathcal{O}(n^{3+\varepsilon}) \\
		m &\leq \Omega(c^{5} n^{5}) = \Omega(n^{4+\varepsilon}) \label{eq:n_2-constraint}
	\end{align}
	With this trade-off, the second constraint, \eqref{eq:n_2-constraint}, will eventually be irrelevant for large $ n $. 
\end{remark}

\begin{theorem}\label{thm:gmm-identifiable-deg6}
	$ \GM_{6}(\C^n) $ is $ m $-identifiable for all $ m \in \N $ bounded by some function $ m\le  \binom{n + 5}{6}/\dim \GM_{6}(\C^n) - \mathcal{O}(n^{3.00001}) $. 
\end{theorem}
\begin{proof} 
	The third condition of \Cref{thm:massarenti-mella} is satisfied due to \Cref{lem:gm-deg6-skewtangents}, combined with  \Cref{rem:gm6ident-splitting-strat}. \Cref{prop:gm6-not-1-tangentially-weakly-defective} verifies Condition (2) of \Cref{thm:massarenti-mella}. Condition (1) of \Cref{thm:massarenti-mella} is trivially satisfied for our choices of $ m $. 
\end{proof}

As a direct consequence, we obtain the following result in statistical language. 
\begin{theorem}\label{thm:main-result}
	For some $ m=\Theta(n^4) $, general weights $ \lambda_1,\ldots,\lambda_m\in \R_{>0} $, general linear forms $ \mu_1,\ldots,\mu_m \in \R^n $ and general positive definite covariance matrices $ \Sigma_1,\ldots,\Sigma_m\in \R^{n\times n} $, denote by 
	\begin{align}\label{eq:gm-mu-sigma}
		f_{6} = \lambda_1 s_6(\mu_1, \Sigma_1)  + \ldots + \lambda_m s_6(\mu_m, \Sigma_m)\\
		f_{4} = \lambda_1 s_4(\mu_1, \Sigma_1)  + \ldots + \lambda_m s_4(\mu_m, \Sigma_m)\nonumber
	\end{align}
	the moment forms of degree $ 4 $ and $ 6 $, respectively, of the Gaussian mixture that is parametrized by the $ \mu_i $ and $ \Sigma_i $. 
	
	Then, there is only one way to represent $ (f_{4}, f_{6}) $ as the mixture moments of (at most) $ m $ Gaussians. Precisely, if $ \nu_1,\ldots,\nu_m \in \R^n $, $ \rho_1,\ldots,\rho_m\in \R_{\ge 0} $ and $ T_1,\ldots,T_m\in \R^{n\times n} $ are symmetric matrices such that 
	\begin{align}
		f_{6} = \lambda_1 s_6(\nu_1, T_1)  + \ldots + \lambda_m s_6(\nu_m, T_m)\\
		f_{4} = \lambda_1 s_4(\nu_1, T_1)  + \ldots + \lambda_m s_4(\nu_m, T_m)\nonumber
	\end{align}
	then, up to permutation, $ \lambda_i = \rho_i $, $ \mu_i = \nu_i $, and $ \Sigma_i = T_i $, for all $ i\in \{1,\ldots,m\} $. 
	Additionally, if the mixing weights $ \lambda_i = \frac{1}{m} $ are uniform, then the same result holds true without any requirement about $ f_{4} $.
\end{theorem}
\begin{proof}
	For the uniformly weighted case, the claim directly follows from \Cref{thm:gmm-identifiable-deg6}. \Cref{lem:reduce-weighted-to-uniform} then yields the weighted case. 
\end{proof}

\section{Numerical experiments, degree $ 5 $ and inhomogeneous moments}\label{sec:numerics}

Numerical results suggest that the Gaussian moment varieties $ \GM_d(\C^n) $ behave in a very regular way with respect to identifiability: It appears that for the Gaussian moment varieties of degree $ 5 $ and $ 6 $, all secants are nondefective up to the rank bound obtained from counting parameters. The Gaussian moment variety of degree $ 4 $ is never identifiable, for simple reasons explained in \Cref{sec:deg4-nonident}. We formalize the numerical results in the Conjectures \ref{conj:deg5-ident} and \ref{conj:gm4}. Code and data of the dimension calculations may be found in \cite[\texttt{secant-dimensions}]{Taveira_Blomenhofer_Code_Gaussian_Identifiability_2023}. 

\begin{conjecture}\label{conj:deg5-ident}
	For all $ m, n, d\in \N $ with $ n\ge 2 $ and $ d\ge 5 $, the Gaussian moment variety $ \GM_{d}(\C^n) $ is $ m $-identifiable, if 
	\begin{align}
		m < \dfrac{\binom{n + d - 1}{d}}{\binom{n+1}{2}} - 1. 
	\end{align}
\end{conjecture}

The numerical experiments were conducted by sampling random linear forms and covariance matrices $ (\ell_{1},q_1),\ldots,(\ell_{m},q_m) $, and then calculating the dimension of the sum of tangent spaces $ \sum_{i = 1}^{m} T_{\ell_{i}, q_i} \GM_{d}(\C^n) $, which is the third respective column in \Cref{table:gm-ident-deg-5-6}. The column labeled ``expected dimension'' is the dimension of the direct sum of tangent spaces. We always chose the rank $ m $ in accordance with the bound from counting parameters.

\begin{table}[!h]
	\centering
	\begin{minipage}[t]{0.44\linewidth}
		\hspace{0.5em} \textbf{$ \GM_{5}(\C^n) $-secants}\hspace{0.5em} 
		\begin{filecontents*}{data/secant-dimensions-deg5.csv}
"n","secant dimension","expected dimension","rank"
2,5,5,1
3,18,18,2
4,56,56,4
5,120,120,6
6,243,243,9
7,455,455,13
8,792,792,18
9,1242,1242,23
10,1950,1950,30
11,3003,3003,39
12,4320,4320,48
13,6136,6136,59
14,8568,8568,72
15,11610,11610,86
16,15504,15504,102
17,20230,20230,119
18,26271,26271,139
19,33649,33649,161
20,42320,42320,184
\end{filecontents*}
\csvreader[tabular=rrrr,
		table head=  n & rank & secant dim. & exp. dim. \\ \hline, late after line=\\]
		{data/secant-dimensions-deg5.csv}{"n"=\n,"secant dimension"=\sdim,"expected dimension"=\edim, "rank"=\rank}
		{\n & \rank & \sdim & \edim }
	\end{minipage} \vline
	\begin{minipage}[t]{0.46\linewidth}
		\hspace{0.5em}\textbf{$ \GM_{6}(\C^n) $-secants}
		\begin{filecontents*}{data/secant-dimensions-deg6.csv}
"n","secant dimension","expected dimension","rank"
2,5,5,1
3,27,27,3
4,84,84,6
5,200,200,10
6,459,459,17
7,910,910,26
8,1716,1716,39
9,2970,2970,55
10,5005,5005,77
11,8008,8008,104
12,12330,12330,137
13,18512,18512,178
14,27132,27132,228
15,38745,38745,287
16,54264,54264,357
17,74460,74460,438
18,100926,100926,534
19,134596,134596,644
\end{filecontents*}
\csvreader[tabular=rrrr,
		table head=  n & rank & secant dim. & exp. dim. \\ \hline, late after line=\\]
		{data/secant-dimensions-deg6.csv}{"n"=\n,"secant dimension"=\sdim,"expected dimension"=\edim, "rank"=\rank}
		{\n & \rank & \sdim & \edim }
	\end{minipage}
	\caption[Identifiability for $ \GM_{5}(U) $]{Numerical experiments show that for $ n=2,\ldots, 19 $, the Gaussian moment variety is nondefective up to the maximum possible rank $ m = \lfloor \dim \C[X]_{d}/\dim \GM_{d}(\C^n) \rfloor $ for both $ d = 5 $ (left) and $ d = 6 $ (right). The value for $ (d, n)=(6, 20) $ is missing due to memory limitations. }
	\label{table:gm-ident-deg-5-6}
\end{table}

We formalize the numerical results from \Cref{table:gm-ident-deg-5-6} in the following theorem. 
\begin{theorem}
	For any $ n \in \{1,\ldots,19\} $ and $ d\in \{5,6\} $, if $ m $ is strictly smaller than the value in the ``rank'' column of \Cref{table:gm-ident-deg-5-6}, then $ \GM_{d}(\C^n) $ is $ m $-identifiable. In addition, 
	\begin{enumerate}
		\item for all $ n\ge 20 $, $ \GM_{5}(\C^n) $ is $ m $-identifiable for all $ m\leq 183 $.
		\item for all $ n\ge 19 $, $ \GM_{6}(\C^n) $ is $ m $-identifiable for all $ m\leq 643 $.
	\end{enumerate}
\end{theorem}
\begin{proof}
	First, for $ n\in \{1,\ldots,19\} $, the claims follow from \Cref{table:gm-ident-deg-5-6} together with \Cref{thm:massarenti-mella} and \Cref{prop:gm6-not-1-tangentially-weakly-defective}. Note that we require $ m $ to be \emph{strictly} smaller than the parameter counting bound  $ \dim \C[X]_{d}/\dim \GM_{d}(\C^n) $, in order to account for the conditions in \Cref{thm:massarenti-mella}. Also, note that if an irreducible, nondegenerate affine cone $ V $ is $ m $-identifiable, then it is also $ (m-1) $-identifiable, see \cite[Lemma 2.2.32]{Taveira_Blomenhofer_Thesis}. Thus, identifiability holds for all values $ m $ which are smaller than the value in the ``rank'' column of \Cref{table:gm-ident-deg-5-6}. (NB: Obviously, $ m $-nondefectivity also implies $ (m-1) $-nondefectivity). 
	
	Last, if $ \GM_{d}(\C^{n_0}) $ is $ m $-identifiable (or $ m $-nondefective) for some $ n_0\in \N $, then  $ \GM_{d}(\C^{n}) $ is $ m $-identifiable (or $ m $-nondefective, respectively), for all $ n\ge n_0 $: Indeed, consider a general linear projection $ \pi\colon \C^n \to \C^{n_0}$ and the induced map $ \Pi\colon \C[X] \to \C[X_1,\ldots,X_{n_0}] $. Observe that if a sum of $ m $ tangent spaces to $ \GM_{d}(\C^n)$ was not direct, then their projection to tangents of $ \GM_{d}(\C^{n_0}) $ would not be direct, either. Similarly, if a general element $ t \in \sigma_m \GM_{d}(\C^n) $ had two different decompositions as $ m $-fold sums, then $ \Pi(t) $ would have, too. 
\end{proof}

\begin{figure}[h]
	\centering
	\def\svgwidth{0.7\columnwidth}
	\includesvg{secant-dimensions.svg}
	\caption{The blue, red, green points correspond to the dimensions of the secant variety $ \sigma_m \GM_d(\C^n) $ for $ d = 4,5,6 $, respectively, and several small values of $ n $. We always choose $ m = \lfloor \frac{\dim \C[X]_{d}}{\binom{n+1}{2} + n} \rfloor $. This is the bound for identifiability obtained from counting parameters. The dashed lines correspond to the expected dimensions. }
	\label{fig:secant-dimensions}
\end{figure}

\subsection{Non-identifiability from degree-$ 4 $ moments}\label{sec:deg4-nonident}

Moments of degree at most $ 4 $ never suffice to recover the parameters of a mixture of at least $ 2 $ Gaussians. In fact, a mixture of two general Gaussians will have infinitely many Gaussian mixture representations of rank $ 2 $ in degree $ 4 $. To see this, look at the tangent space 
\begin{align}
	T_{\ell_{1}, q_1}\GM_{4}(\C^n) + T_{\ell_{2}, q_2}\GM_{4}(\C^n), 
\end{align}
which consists of elements of the form
\begin{align}
	(\ell_1^3 + q_1\ell_1) h_1 + (\ell_1^2 + q_1) p_1 + (\ell_2^3 + q_2\ell_2) h_2 + (\ell_2^2 + q_2) p_2, 
\end{align}
where $ h_1, h_2 $ are linear forms and $ p_1, p_2 $ are quadratic forms. The subspaces spanned by the $ p_1 $- and the $ p_2 $-expressions have a nonempty intersection: Indeed, note that for $ p_1 = (\ell_2^2 + q_2) $ and $ p_2 = -(\ell_1^2 + q_1) $, it holds that 
\begin{align}\label{eq:koszul-syzygy}
	(\ell_1^2 + q_1) p_1 + (\ell_2^2 + q_2) p_2 = 0.
\end{align}
Therefore, $ \sigma_2 \GM_{4}(\C^n) $ is not of expected dimension. The alert reader might have noticed that \eqref{eq:koszul-syzygy} is a Koszul syzygy of the ``shifted'' covariance matrices $ \ell_{i}^2 + q_i $ (which are the second-order moment forms). Thus, the defect of the $ m $-th secant of $ \GM_{4}(\C^n) $ is at least $ \binom{m}{2} $.

Note that the dashed lines in \Cref{fig:secant-dimensions} correspond to the expected dimensions, which, for degree $ 5 $ and $ 6 $, agree with the actual dimensions (shown as points in \Cref{fig:secant-dimensions}). However, for degree $ 4 $, this is not the case. Here, the blue points align with the blue non-dashed line, which plots the expected dimension minus $ \binom{m}{2} $. 
Thus, our numerical experiments suggest the following conjecture. 

\begin{conjecture}\label{conj:gm4}
	For $ m, n\in \N_{\ge 2} $, the (tangential) defect of $ \sigma_m \GM_{4}(\C^n) $ is precisely $ \binom{m}{2} $, unless $ \sigma_m \GM_{4}(\C^n) $ fills the entire space $ \C[X]_{4} $. In other words, the defect of the $ m $-th secant of $ \GM_{4}(\C^n) $ is completely explained by the Koszul syzygies of the second-order moments. 
\end{conjecture}

It is instructive to compare \Cref{conj:gm4} with Ottaviani's conjecture on powers-of-forms, as described in \cite[Conjecture 1.2]{Lundqvist_Oneto_Reznick_Shapiro_2019}, or, independently, also in \cite[Conjecture 1]{Hieu_Sorber_VanBarel_2013}: For sums of squares, Ottaviani's conjecture states that the defect of the $ m $-th secant variety of squares of $ k $-forms is precisely $ \binom{m}{2} $. We believe that the first stepping stone towards resolving \Cref{conj:gm4} should be to work on this particular case of Ottaviani's conjecture.

\section{Generic ranks of Gaussian moment varieties}

At last, we would like to address another classical question for the setting of Gaussian moments: Given a generic $ d $-form $ f $, how many generalized Gaussian moments (i.e., elements of $ \GM_{d}(\C^n) $) do we need to represent $ f $ as their sum? An analogous question was classically studied for Waring decompositions and resulted in the celebrated Alexander-Hirschowitz Theorem, see \cite{hirschowitz1995polynomial}. 

\begin{defi}
	Let $ V $ an irreducible, nondegenerate affine cone in some finite-dimensional $ \C $-vector space $ U $. Then, the smallest integer 
	$ m \in \N_0 $ such that $ \sigma_m V = U $ is called the \emph{generic rank} with respect to $ V $. 
\end{defi}

\begin{prop}
	Let $ n, d\in \N $ with $ d\ge 5 $ and denote by $ m^{\circ} $ the generic rank with respect to $ \GM_{d}(\C^n) $. Then, 
	\begin{align}
		\left\lceil\frac{\binom{n+d-1}{d}}{\binom{n+1}{2}+n}\right\rceil \le m^{\circ} \le \left\lceil \frac{\binom{n+d-1}{d}}{\binom{n+1}{2}} + \binom{n+1}{2} \right\rceil
	\end{align}
\end{prop}
\begin{proof}
	The lower bound stems from counting parameters: Clearly, the summation map $ \GM_{d}(\C^n)^{m} \to \C[X]_{d} $ can only be dominant, if the dimension of the domain is at least the dimension of the codomain. 
	The upper bound is due to the following observation: The dimension of the secant variety $ \sigma_m \GM_{d}(\C^n) $ equals the dimension of its tangent space at a general point. Let $ (\ell_1,q_1),\ldots,(\ell_m, q_m) $ pairs of general linear and quadratic forms on $ \C^n $.  The tangent space to the secant variety at $ f = s_{d}(\ell_1,q_1) + \ldots + s_{d}(\ell_m, q_m) $ equals the set of all expressions
	\begin{align}
		s_{d-1}(\ell_i, q_i)h_i + s_{d-2}(\ell_i, q_i)p_i, \qquad (h_i\in \C[X]_{1}, p_i\in \C[X]_{2})
	\end{align}
	A lower bound for the dimension of this space can therefore be obtained by setting all $ h_i $ to zero. Then, the set 
	\begin{align}
		\{s_{d-2}(\ell_i, q_i)p_i \mid p_i\in \C[X]_{2} \} = (s_{d-2}(\ell_{1},q_1),\ldots,s_{d-2}(\ell_{m},q_m))_{d}
	\end{align}
	is a graded component of the ideal generated by the forms $ s_{d-2}(\ell_i, q_i) $. We can now apply Nenashev's result, but this time, we use the second case of \Cref{thm:nenashev}, applied to the variety $ \mathcal{D} := \im s_{d-2} $, which is closed under the action of $ \GL(\C^n) $. We obtain from \eqref{eq:nenashev-2} in \Cref{thm:nenashev}, that for 
	\begin{align}
		m\ge \frac{\binom{n+d-1}{d}}{\binom{n+1}{2}} + \binom{n+1}{2},
	\end{align} 
	the ideal spanned by the $ s_{d-2}(\ell_i, q_i) $ is the entire space $ \C[X]_{d} $. 
\end{proof}

\section{Conclusions}

\subsection{Outlook: Even degree identifiability}\label{sec:deg-even-ident}

It is possible to verify that the proof of \Cref{lem:gm-deg6-skewtangents} generalizes to several higher even degrees $ d\ge 6 $. While the structure of the proof remains the same, the occurring coefficients change. It would therefore be good to know whether \Cref{thm:main-result} generalizes to \emph{all} even degrees $ d\ge 6 $ simultaneously. This would take some significant combinatorial effort. Similarly, in order to generalize \Cref{prop:gm6-not-1-tangentially-weakly-defective} to higher degrees, one would need to replace the computer argument. We leave this to future work. 

\subsection{Outlook: Odd degree identifiability}\label{sec:deg-odd-ident}

We also expect the moment varieties of odd degree to be identifiable. A proof that a sum of $ m=\Theta(n^4) $ general tangents to $ \GM_{7}(\C^n) $ is skew can be found in my doctoral thesis \cite[Theorem 3.3.8]{Taveira_Blomenhofer_Thesis}. Unfortunately, for odd degrees, the variable splitting argument loses a multiplicative factor of $ n $, and therefore does not match the upper bound from counting parameters, which would be $ m=\Theta(n^5) $, for degree $ 7 $. 

\begin{cor}
	$ \GM_{7}(\C^n) $ is $ m $-identifiable for some $ m=\Theta(n^4) $. 
\end{cor}
\begin{proof}
	Combining \cite[Theorem 3.3.8]{Taveira_Blomenhofer_Thesis} with \Cref{prop:gm6-not-1-tangentially-weakly-defective} shows that the assumptions of \Cref{thm:massarenti-mella} are satisfied. 
\end{proof}

\subsection{Outlook: Exponential varieties}
\raggedbottom
The Veronese varieties, powers-of-forms varieties (as considered in \cite[Section 4]{Blomenhofer_Casarotti_Michalek_Oneto_2022}) and Gaussian moment varieties all fall into a similar category of varieties, which we aim to capture with the following definition. 

\begin{defi}
	Let $ U $ a finite-dimensional subspace of the ideal $ (X_1,\ldots,X_n) $. Then for $ d\in \N $, we call the closure of the image of 
	\begin{align}
		s_{d}(U)\colon U\to \C[X]_{d}, f\mapsto \exp(f)_{=d},
	\end{align}
	the degree-$ d $ \emph{exponential variety} $ \mathcal{E}_{d}(U) $ of $ U $. Here, $ g_{=d} $ denotes the degree-$ d $ homogeneous part of some power series $ g $ in $ X $. 
\end{defi}
The Veronese varieties are recovered, if one takes $ U $ as the space of linear forms. Powers-of-forms varieties are recovered, if one takes $ U $ as the space of $ k $-forms, for some $ k\in \N $. The Gaussian moment variety is recovered by this notion, if one takes  $ U = \C[X]_{1} \oplus \C[X]_{2} $. 
It is natural to conjecture that many other exponential varieties will eventually have identifiable secants, if $ d $ is large enough. E.g. for $ U = \C[X]_{1} \oplus \ldots \oplus \C[X]_{k} $, one expects a similar variable-splitting argument to work for large $ d $, unless something goes wrong with the combinatorics of the coefficients of $ s_{d}(U) $. Identifiability of exponential varieties could therefore be an interesting direction of future study, and we hope that the present work gave some convincing motivation for it.

\appendix

\bibliography{bibML}
\bibliographystyle{acm}
\end{document}